\documentclass{amsart}

\usepackage{amsmath,amsthm,amsfonts,eucal}

\usepackage{latexsym}
\usepackage{amssymb}

\usepackage[active]{srcltx}
\usepackage[dvips]{epsfig}

\usepackage[usenames,dvipsnames]{xcolor}
\usepackage[pdfusetitle,
bookmarks=true,bookmarksnumbered=true,bookmarksopen=true,bookmarksopenlevel=1,
breaklinks=false,pdfborder={0 0
  0},backref=false,colorlinks=true,linkcolor=NavyBlue,citecolor=NavyBlue,urlcolor=black]
{hyperref}
\pdfpageheight\paperheight
\pdfpagewidth\paperwidth

\usepackage{enumitem}
\renewcommand{\Im}{\operatorname{Im}}

\newcommand{\C}{\mathbb{C}}
\renewcommand{\epsilon}{\varepsilon}
\renewcommand{\Im}{\operatorname{Im}}
\usepackage[capitalise]{cleveref}
\crefformat{equation}{(#2#1#3)}

\let\deg\relax
\DeclareMathOperator{\deg}{deg}

\DeclareMathOperator{\reg}{reg}
\DeclareMathOperator{\sng}{sng}

\newcommand{\cX}{\mathcal{X}}
\newcommand{\cV}{\mathcal{V}}

\newcommand{\fe}{\mathfrak{e}}

\newcommand{\uv}{\underline v}
\newcommand{\um}{\underline m}

\newcommand{\car}{\mathop{\rm{Car}}\nolimits}

\newcommand{\Car}{\mbox{\rm Car}}

\def\cZ{{\mathcal Z}}

\newcommand{\cB}{{\mathcal{B}}}
\newcommand{\cD}{{\mathcal{D}}}

\newcommand{\cN}{{\mathcal{N}}}

\newcommand{\cS}{{\mathcal{S}}}

\newcommand{\cL}{{\mathcal{L}}}
\newcommand{\cP}{{\mathcal{P}}}

\newcommand{\IC}{{\mathbb{C}}}

\newcommand{\uz}{{\underline{z}}}
\def\beeq{\begin{equation}}
\def\eneq{\end{equation}}
\newcommand{\Res}{\mathop{\rm{Res}}}

\newcommand{\bse}{\begin{equation}}

\newcommand{\vvec}{{\underline v}}

\newcommand{\be}{\begin{eqnarray}}
\newcommand{\ee}{\end{eqnarray}}

\newcommand{\tor}{{\mathbb T}}
\newcommand{\Z}{{\mathbb Z}}

\newcommand{\mes}{{\rm mes}}

\newtheorem{theorem}{Theorem}[section]
\newtheorem{thma}{Theorem}

\newtheorem{thmb}{Theorem}

\newtheorem{thmc}{Theorem}

\newtheorem{lemma}[theorem]{Lemma}

\newtheorem{prop}[theorem]{Proposition}

\theoremstyle{definition}
\newtheorem{defi}[theorem]{Definition}
\newtheorem{example}[theorem]{Example}

\theoremstyle{remark}
\newtheorem{remark}[theorem]{Remark}

\numberwithin{equation}{section}

\begin{document}
\title
[Cartan Covers and Doubling Bernstein Type Inequalities]
{Cartan Covers and Doubling Bernstein Type Inequalities
on Analytic Subsets of $\mathbb{C}^2$}
\date{}
\author{ Michael Goldstein, Wilhelm Schlag, Mircea Voda}

\address{Department of Mathematics, University of Toronto,
  Toronto, Ontario, Canada M5S 1A1}
\email{gold@math.toronto.edu}

\address{Department of Mathematics, The University of Chicago, 5734 S. University Ave., Chicago, IL 60637, U.S.A.}
\email{schlag@math.uchicago.edu}

\address{Department of Mathematics, The University of Chicago, 5734 S. University Ave., Chicago, IL 60637, U.S.A.}
\email{mvoda@uchicago.edu}

\thanks{The second author was partially supported by the NSF, DMS-1500696. The first author thanks the University of Chicago for its hospitality during the months of July and August of 2016. The authors are grateful to J\'anos Koll\'ar and Mihnea Popa for helpful discussions on B\'ezout's theorem. }

\begin{abstract}
We prove a version of the
doubling Bernstein inequalities
for the trace of an analytic function of two variables on an analytic subset of $\mathbb{C}^2$.
The estimate applies to the whole analytic set in question
including its singular points.
The proof relies on a version of the Cartan estimate for maps in $\mathbb{C}^2$ which we establish in this work.
\end{abstract}

\maketitle\tableofcontents


\section{Introduction}\label{Sec:intro}
In a series of papers \cite{FefNar93},\cite{FefNar94},\cite{FefNar96}, Fefferman and Narasimhan investigated the local behavior of a
polynomial $f$ of $N$ real or complex variables, restricted to a  given
$n$-dimensional algebraic variety $\cX$.  Conceptually, the problem is to quantify to what extent
 the local behavior of the trace of $f$ on $\cX$ deteriorates relative to
an  $N$-dimensional ball. Of particular interest here is to determine the
dependence of quantitative estimates on the degree of the
polynomials.
Fefferman and Narasimhan chose the classical Bernstein inequalities for polynomials of several variables
to measure the distortion of a polynomial restricted to an algebraic variety.

The authors' interest in this particular problem arose as part of their work on the  Chulaevsky-Sinai conjecture.
In their pioneering paper \cite{ChuSin89}, Chulaevsky and Sinai  analyze the  spectrum of a discrete Schr\"odinger operator on $\Z$ with a quasi-periodic potential
given by evaluating a generic smooth function on $\tor^{2}$ along the orbit of an ergodic shift.
In \cite{GolSchVod16a} (building on work from \cite{GolSchVod16}) the authors found  that some ``generic
versions" of these restricted Bernstein estimates play a crucial role in addressing this conjecture.

\smallskip

There are two major differences between the current paper and \cite{FefNar96}: (i) we obtained estimates at
singular points and the estimates at regular points don't depend on the distance to the singular points
(ii) we allow analytic functions and analytic sets in place of polynomials and algebraic varieties.

As for (i),  Fefferman and Narasimhan had considered compact subsets of algebraic varieties 
away from the singular points. For polynomials and algebraic varieties, Roytwarf and Yomdin~
\cite{RoyYom97} extended their Bernstein estimates to be independent of the distance to the singular points.
However, the aforementioned spectral analysis forces us to consider analytic functions and sets, rather than algebraic ones.
Our estimates for analytic functions are not as sharp as for
polynomials.  

Regarding (ii), we note that Coman-Poletsky \cite{ComPol07} (for $n=1$) and Brudnyi
\cite{Bru08} (for all $n\ge1$) studied Bernstein estimates (amongst other local properties)
for the restriction of analytic functions of $n+1$ variables to the graph of an analytic function of $n$ variables.
Both these papers require a certain {\em transversality condition} of the
zeros sets of the functions in question.  We shall also impose a condition of this nature in our approach.

We proceed to discuss the main results of the paper. First we need to introduce some notation related to Cartan sets
and to Bernstein exponents. The Cartan sets will appear in our transversality condition to allow the application
of the Cartan-type estimate established in \cref{thm:C}.

\begin{defi}\label{defi:cartan0sets} $(1)$ Let $H\ge 0$, $ K\ge 1 $.
  For $\cB \subset\IC^2$ we say that $\cB \in \car_{2,0}(H, K)$ if
  \[
    \cB\subset \bigcup\limits^{j_0}_{j=1} B(\uv_j, r)
  \]
  with $r=e^{-H}$ and $j_0 \le K$.

  $(2)$ Let $ f$ be analytic on the ball
  $B(\vvec_0,R)\subset \mathbb{C}^{2}$, $ \cS\subset \C^2 $, and $ \mu\in(0,1) $. Define
  \begin{gather*}
      M_f(\vvec_0,R)=\sup_{B(\vvec_0,R)}\log |f|,\quad
      M_f(\cS,\vvec_0,R)=\sup_{B(\vvec_0,R)\cap \cS}\log |f|,\\
      B_f(\mu;\vvec_0,R) =M_f(\vvec_0,R)-M_f(\vvec_0,\mu R),\\
      B_f(\mu;\cS,\vvec_0,R) =M_f(\cS,\vvec_0,R)-M_f(\cS,\vvec_0,\mu R).
  \end{gather*}
  We call $B_f(\mu;\vvec_0,R),B_f(\mu;\cS,\vvec_0,R)$ {\em Bernstein exponents}. We make the natural
  convention that if the function $ f $ vanishes identically, its Bernstein exponents are zero.

  $ (3) $ Let $ f $ be analytic on $ B(0,1) $, $ \mu\in(0,1) $. We define
  \begin{equation*}
    B_f(\mu)=\sup_{\uv_0\in B(0,1/4),0<R\le 1/4} B_f(\mu;\uv_0,R).
  \end{equation*}

  $ (4) $ Given an analytic function $ f $ on a disk $ \cD(z_0,R)\subset \C $, the quantities
  $ M_f(z_0,R) $ and $ B_f(\mu; z_0,R) $ are defined analogously to the above.
\end{defi}

\smallskip
The classical Bernstein doubling inequality for a univariate polynomial $ f $ can be expressed using the above
notation as 
\begin{equation*}
  B_f(\mu;z_0,R)\le (\log \mu^{-1})\times \deg f,
\end{equation*}
where $ \mu\in (0,1) $, $ z_0\in \C $, $ R>0 $.

\smallskip

Throughout we will impose the following {\em transversality condition}.  Suppose
the functions $f_1,f_2$ are analytic in the ball $B(0,1)\subset \C^2$,  and are
normalized so that $ M_{f_i}(0,1)\le 0 $, $ i=1,2 $. We let $ F=(f_1,f_2) $ and
we define
\begin{equation*}
  \cN_F(\epsilon):= \{ \uv\in B(0,1): |F(\uv)|<\epsilon \}.
\end{equation*}
We require that
\begin{equation}\label{eq:cartanmaincond}
  \cN_F(\exp(-H_0))\cap B(0,1/2)\in \car_{2,0}(H_1, K_1),\quad \log K_1\ll H_1,
\end{equation}
for some $H_0\gg H_1\gg B_0:=\max_i B_{f_i}(1/4)$.

\begin{remark}
  A priori it might appear that $ K_1 $ can be exponentially large, i.e.,
  $ \exp(cH_0) $ for some small $ c>0 $. However, a simple argument,
  presented in \cref{lem:K1-sa-bound}, shows that we always have the polynomial bound
  $ K_1\le H_0^C $, where $ C $ is some absolute constant.
\end{remark}

Let $\cZ=\{\uv\in B(0,1): f_2(\vvec)=0\}$.  It is well known that there exists a discrete set of singular points
$ \sng \cZ $ (relative to $ B(0,1) $) such that the set of regular points $ \reg \cZ:= \cZ\setminus \sng \cZ $ is
a one dimensional complex manifold (see, for example, \cite{Chi89}).

\begin{thma}\label{thm:A}
  Assume the transversality condition holds and let $ \cZ $ be as above.
  Let $ C_0=\log (K_1B_0^2H_0^2) $. Then the following statements hold.
  \begin{enumerate}[leftmargin=*]
  \item
    For any $ \uv_0\in B(0,1/8)\cap \cZ $ and $ 0<R\le 1/4 $,
    \begin{equation*}
      B_{f_1}(1/4;\cZ,\uv_0,R)\lesssim \max(\log R^{-1},C_0)B_0^2H_0.
    \end{equation*}
  \item
      There exists an atlas of $  \reg \cZ $ with charts defined on $ \cD(0,1) $ such that for any chart
      $ \phi $ satisfying $ \phi(\cD(0,1))\cap B(0,1/8)\neq \emptyset $ and any $ \cD(z_0,R)\subset \cD(0,1) $,
      we have
      \begin{equation*}
        B_{f_1\circ \phi}(1/4;z_0,R)\le C(f_2)C_0B_0^2H_0.
      \end{equation*}
  \end{enumerate}
\end{thma}

\begin{remark}
  The $ \log R^{-1} $ factor from part (1) of \cref{thm:A} is needed because the estimate covers singular points.
  See \cref{ex:logR}.
\end{remark}

In \cref{thm:B} we obtain a sharper version of part (2) of the previous theorem for the polynomial case.
Such a result is also known from \cite{RoyYom97}. The work of Roytwarf and Yomdin relies on a
classical inequality for the Taylor coefficients of  $p$-valent functions due to Biernacki~\cite{Bie36}.
In turn \cite{Bie36} relies on a deeper growth bound for $p$-valent functions obtained by Cartwright~\cite{Ca35}
(see \cite{Hay94} for a more detailed account of these issues).
In \cref{thm:B} we show that in the context of algebraic curves the Bernstein estimates by
Roytwarf and Yomdin follow from more elementary arguments in the spirit of the argument principle,
without any reference to properties of $p$-valent functions. 
We also require basic properties of the harmonic conjugate and of course  
Bezout's theorem (which is also needed in Roytwarf and Yomdin in order to estimate the valency).
It seems that this approach can be developed for a general algebraic variety.

\begin{thmb}\label{thm:B} Assume that $f_1,f_2$ are polynomials. Let $ \cZ $ be as above.
  Then there exists an atlas of $  \reg \cZ $ with charts defined on $ \cD(0,1) $ such that for any chart
  $ \phi $  and any $ \cD(z_0,R)\subset \cD(0,1) $,  we have
  \begin{equation*}
    B_{f_1\circ \phi}(1/4;z_0,R)\le C(f_2)\times \deg f_1.
  \end{equation*}
\end{thmb}

For our application in \cite{GolSchVod16a} we use the Cartan estimate for maps in
$\mathbb{C}^2$ which is \cref{thm:C} we state below. The proof of \cref{thm:A} relies on
\cref{thm:C}. The Cartan estimate for an analytic
function $f(\vvec)$, $\vvec\in \mathbb{C}^2$ (see \cref{lem:high_cart}), basically says that if
the set $\{|f|<\epsilon_0\}$ is ``not two-dimensional" then
$\{|f|<\epsilon\}$ is ``one-dimensional" for any
$\epsilon\ll \epsilon_0$. We prove an analogue statement for mappings. Let
$F: B(0,1)\subset \mathbb{C}^2\to \mathbb{C}^2$ be analytic.  We show  that if the set $\{|F|<\epsilon_0\}$ is
``zero-dimensional", then $\{|F|<\epsilon\}$ is ``zero-dimensional" for
any $\epsilon\ll \epsilon_0$.  Of course, the quantitative details of
the statement here are as important as the topological ones.

\begin{thmc}\label{thm:C}
  Assume the transversality condition holds. Then for any $H\gg 1$ we have
  \begin{equation*}
    \cN_F(\exp(-HB_0^2H_0))\cap B(0,1/4)\in \car_{2,0}(H, K),\quad K\lesssim K_1 B_0^2H_0^2.
  \end{equation*}
\end{thmc}

The proof of Theorem~\ref{thm:C} proceeds in four steps: (a) apply the Weierstrass preparation theorem to the given analytic functions in one of the two coordinates (b) determine the resultant of the two polynomials obtained in the previous step (c) apply Cartan's theorem in one variable so as to guarantee that this resultant is not too small off of a union of small disks in $\C$, which in turn gives that at least one of the two analytic functions is not too small outside of thin cylinders in $\C^2$ (d) repeat the previous steps with respect to  the other variable. The intersection of the two families of thin cylinders gives a $\car_{2,0}$ set. 

It would be interesting to extend this method   to higher dimensions, i.e., to the construction of $\car_{d,0}(H, K)$ sets with $d\ge3$ -- at least for polynomials in $d$ variables. In principle, this appears possible but it seems to require the use of multivariate resultants, which are more delicate than the univariate ones. If Theorem~\ref{thm:C} extends to $d\ge 3$, then one would obtain a Bernstein estimate as in Theorem~\ref{thm:A}.
As our applications do not require this extension, we do not pursue these matters here.

We conclude this introduction by providing some details of the aforementioned spectral theory applications.
Consider a trigonometric polynomial of two variables
\begin{equation}\label{Vnondegen53}
  \begin{split}
    V(z,w)=\sum_{|m|,|n|\le k} c_{m,n}e(mz+nw),
  \end{split}
\end{equation}
$e(\zeta):=e^{2\pi i \zeta}$.  To normalize the setting we consider
the unit sphere in the space of the coefficients
\[
  \mathcal{C}_1=\{(c_{m,n})\in\mathbb{R}^{4k+2}:\sum_{m,n}|c_{m,n}|^2=1\}.
\]
We use $\mes$ for the Lebesgue measure on the sphere.  Take arbitrary
$\omega \in \tor^2$, $\lambda\in \mathbb{R}$. Consider the determinant
\begin{equation}\label{dirichlet1}
  \begin{split}
    \begin{aligned}
      &f_N(\vvec)=
      \begin{vmatrix}
        \lambda V(\vvec) & -1 & 0  &\cdots &\cdots & 0\\[5pt]
        -1 & \lambda V(\vvec+\omega)& -1 & 0 & \cdots & 0\\[5pt]
        \vdots & \vdots & \vdots & \vdots & \vdots \\
        &&&&&-1 \\[5pt]
        0 & \dotfill & 0 & -1 && \lambda V(\vvec+(N-1)\omega)
      \end{vmatrix}
    \end{aligned}
  \end{split}
\end{equation}
For $\vvec\in \mathbb{R}^2$,  $f_N(\vvec)$ is the characteristic
determinant of the Schr\"{o}dinger operator with potential
$V(\vvec+n\omega)$, $n\in \mathbb{Z}$ on the interval $[0,N-1]$ subject
to Dirichlet boundary conditions. In \cite{GolSchVod16a} we establish the
following results: {\em Given arbitrary $\epsilon>0$, there exists a set
$\mathcal{C}\subset \mathbb{R}^{4k+2}$ with
$\mes (\mathcal{C}_1\setminus \mathcal{C})<\epsilon$ and
$\lambda_0=\lambda_0(\epsilon)$ depending only on $\epsilon$ such that
for any $V$ with $(c_{m,n})\in \mathcal{C}_1$ and any
$|\lambda|\ge \lambda_0$ there exists
a set $\Omega(V)\subset \tor^2$ with $\mes (\Omega(V))<\epsilon$ such
that for any $\omega\in \tor^2\setminus \Omega(V)$, any $N$, and any
$\vvec_0\in \tor^2$ the functions $f_N(\vvec_0+r_0\vvec)$ and
$f_N(\vvec_n+r_0\vvec)$, $\vvec_n=\vvec_0+n\omega$, $|n|>N$,
$\vvec\in B(0,1)$, $r_0=\exp (-(\log N)^A)$, $A\gg 1$ being an
absolute constant, obey all conditions of \cref{thm:A,thm:C} with
$B_0,H_0\le (\log N)^c$, $c\ll 1$.}

The exceptional sets in this result are not artificial.  In fact,
 the theorem fails for some
$(c_{m,n})\in\mathcal{C}$.  A similar fact is true for the exceptional
frequencies.

\section{Cartan's Estimate}

Recall the following definition from \cite{GolSch08}.

\begin{defi}\label{defi:cartansets}
  Let $H\ge 0$, $ K\ge 1 $.  For an arbitrary set $\cB \subset \IC$
  we say that $\cB \in \car_1(H, K)$ if $\cB\subset
  \bigcup\limits^{j_0}_{j=1} \cD(z_j, r_j)$ with $j_0 \le K$, and
  $\sum_j\, r_j < e^{-H}$.

  If $d\ge 1$ is an integer  and $\cB\subset \IC^d$, then we define
  inductively that $\cB\in \car_d(H, K)$ if for any $1 \le j \le d$ there
  exists $\cB_j \subset \IC, \cB_j \in \car_1(H,  K)$, so that $\cB_z^{(j)} \in \car_{d-1}(H, K)$ for any $z \in \IC
  \setminus \cB_j$,  here $\cB_z^{(j)} = \left\{(z_1, \dots, z_d) \in \cB:
    z_j = z\right\}$.
\end{defi}

The above definition of Cartan sets is motivated by the following statement, known as Cartan estimate on
the lower bound of an analytic function of several variables.

\begin{lemma}[{\cite[Lem.~2.15]{GolSch08}}]\label{lem:high_cart}
 Let $\varphi(z_1, \dots, z_d)$ be an analytic function defined
in a polydisk $\cP = \prod\limits^d_{j=1} \cD(z_{j,0}, 1)$, $z_{j,0} \in
\IC$.  Let $M \ge \sup\limits_{\uz\in\cP} \log |\varphi(\uz)|$,  $m \le \log
\bigl |\varphi(\uz_0)\bigr |$, $\uz_0 = (z_{1,0},\dots, z_{d,0})$.  Given $H
\gg 1$ there exists a set $\cB \subset \cP$,  $\cB \in
\car_d\left(H^{1/d}, K\right)$, $K = C_d H(M - m)$,  such that
\beeq
\label{eq:cart_bd}
\log \bigl | \varphi(z)\bigr | > M-C_d H(M-m)
\eneq
for any $z \in \prod^d_{j=1} \cD(z_{j,0},1/6)\setminus \cB$. Furthermore, when $ d=1 $ we can take
$ K=C(M-m) $ and keep only the disks of $ \cB $ containing a zero of $\phi$ in them.
\end{lemma}

\begin{remark}
  (1)
  The choice of the constant $ 1/6 $ in \cite[Lem.~2.15]{GolSch08} was so that one could invoke the
  one-dimensional Cartan estimate as stated in Theorem 4 of \cite[Lecture 11]{Lev96}. However,
  it is straightforward to adjust the result from \cite{Lev96}
  and the proof from \cite{GolSch08} to replace $ 1/6 $ by any $ r<1 $. Of course, the constant $ C_d $ would depend
  (explicitly) on the particular choice of $ r $.

  (2) The definition of Cartan sets gives implicit information about their measure. For example, using Fubini and
  the definition of $ \Car_d $, one gets by induction that the set exceptional set $ \cB $ in the previous lemma
  satisfies $ \mes_{\C^d}(\cB)\le C(d)\exp(-H) $.
\end{remark}

The following notion will be needed for our discussion of Weierstrass' preparation theorem.
\begin{defi}\label{defi:regulardir}
  Let $f$ be analytic on the ball $B(\vvec_0,R_0)\subset \C^2$.
  Let $\mathfrak{e}\in \mathbb{C}^2$ be an arbitrary unit  vector.
  We say that $\mathfrak{e}$ is $m$--regular for $f$ at $\vvec_0$
  (or just $m$--regular if
  it is clear from the context what $ \uv_0 $ is) if
  \[
    \sup_{z\in\cD(0,R_0/4)}\log |f(\vvec_0+z\mathfrak{e})|\ge m.
  \]
\end{defi}

We show that Cartan's estimate implies that most directions are regular.
We use $ \sigma $ to denote the standard spherical measure.

\begin{lemma}
\label{lem:6.3c}
Let $ f $ be as in \cref{defi:regulardir} and let
\begin{equation*}
  M \ge \sup_{B(\uv_0,R_0)} \log |f|,\quad \sup_{B(\uv_0,R_0/4)}\log |f|\ge m.
\end{equation*}
Take arbitrary $H\gg 1$ and set $\underline{m}= M-C_2H(M-m)$, with $ C_2 $ as in \cref{lem:high_cart}.
Denote by $\cB$ the set of $\mathfrak{e}$ which are not $\underline{m}$--regular.
Then
 \[
\sigma (\cB)\lesssim\exp (-H^{1/2}).
 \]
\end{lemma}
\begin{proof} Apply the Cartan estimate
  to find a set $ \hat \cB $, $\mes (\hat\cB) \lesssim R_0^4\exp (-H^{1/2})$, such that
  $ \log \bigl | f(\vvec)\bigr | >\underline{m} $
for any $\vvec \in B(\vvec_0,R_0/4)\setminus \hat\cB$.
Using spherical coordinates  write
\[
\mes (\hat \cB)\ge\int_{\cB}d\sigma(\mathfrak{e})\int_0^{R_0/4}r^3dr\gtrsim R_0^4\sigma (\cB)
\]
and the statement follows.
\end{proof}

\section{Bernstein Exponent and Number of Zeros}

In this section we provide a relation between Bernstein exponents for one variable analytic functions and the number
of their zeros.
\begin{lemma}\label{lem:Harnack-bounds}
  Let $ \phi $ be a non-vanishing analytic function on $ \cD(z_0,R) $
  Then for  any $ z $, $ |z-z_0|=r<R $, we have
  \begin{equation*}
    -\frac{2r}{R-r}(M-\log|\phi(z_0)|)\le \log|\phi(z)|-\log|\phi(z_0)|
    \le \frac{2r}{R+r}(M-\log|\phi(z_0)|),
  \end{equation*}
  where $ M=M_\phi(z_0,R) $.
\end{lemma}
\begin{proof}
  The estimates follows immediately from Harnack's inequality applied to $ u(z)=M-\log|\phi(z)| $.
\end{proof}

\begin{prop}\label{prop:bernstein-analytic}
  Let $ \phi $ be an analytic function on $ \cD(0,1) $ such that
  \begin{equation*}
    M_\phi(0,1)\le 0,\quad M_\phi(0,1/4)\ge m.
  \end{equation*}
  Let $ n $ be the total number of zeros of $ \phi $ in $ \cD(0,3/4) $.
  Then  for any $ |z_0|<1/8 $, $ r<1/8 $, $ \mu\in (0,1) $, we have
  \begin{equation}\label{eq:B_phi}
    B_{\phi}(\mu;z_0,r)\le Cr(n-m)-n\log\mu\lesssim -(r-\log \mu) m,
  \end{equation}
\end{prop}
\begin{proof}
  Take $\zeta_0\in \cD(0,1/4)$ wit $\log |f(\zeta_0)|= m$. Using Jensen's formula applied to
  \begin{equation*}
    f \left( \frac{z+\zeta_0}{1+\overline {\zeta_0}z} \right)
  \end{equation*}
  we get $ n\lesssim -m $. So, we just have to prove the first estimate in \cref{eq:B_phi}.

  Let $ a_1,\ldots,a_n $, be the zeros of $ \phi $ in $ \cD(0,7/8) $,
  repeated according to their multiplicities.
  Let $ P(z)=\prod_{k=1}^n(z-a_k) $, $ h=\phi/P $, and
  $ z_1 $, $ |z_1-z_0|=\mu r $, be such that $ \log|h(z_1)|=M_h(z_0,\mu r) $.
  Note that  $ h $ is non-vanishing and analytic on $ \cD(0,3/4) $.
  Using \cref{lem:Harnack-bounds}
  we have that for any $ z\in \cD(z_0,\mu r) $
  \begin{multline*}
    \log|h(z)|\ge \log|h(z_1)|-\frac{2|z-z_1|}{1/2-|z-z_1|}(M_h(z_0,1/2)-\log|h(z_1)|)\\
    \ge M_h(z_0,\mu r)-C\mu r(M_h(0,3/4)-M_h(z_0,\mu r)).
  \end{multline*}
  Therefore
  \begin{equation*}
    M_\phi(z_0,\mu r)\ge M_h(z_0,\mu r)-C\mu r(M_h(0,3/4)-M_h(z_0,\mu r))+M_P(z_0,\mu r)
  \end{equation*}
  and
  \begin{multline}\label{eq:B-initial}
    B_\phi(\mu;z_0,r)\le M_h(z_0,r)-M_h(z_0,\mu r)+C\mu r(M_h(0,3/4)-M_h(z_0,\mu r))\\
    +M_P(z_0,r)-M_P(z_0,\mu r).
  \end{multline}
  Let $ z_2 $, $ |z_2-z_0|=r $, such that $ \log|h(z_2)|=M_h(z_0,r) $ and $ z_3 $, $ |z_3|=1/4 $, such that
  $ \log|h(z_3)|=M_h(0,1/4) $. Using \cref{lem:Harnack-bounds} we get
  \begin{multline*}
    M_h(z_0,r)-M_h(z_0,\mu r)=\log|h(z_2)|-\log|h(z_1)|\\
    \le \frac{2|z_2-z_1|}{1/2+|z_2-z_1|}(M_h(z_1,1/2)-\log|h(z_1)|)
    \le Cr(M_h(0,3/4)-M_h(z_0,\mu r)),
  \end{multline*}
  \begin{multline*}
	M_h(z_0,\mu r)-M_h(0,1/4)=\log|h(z_1)|-\log|h(z_3)|\\
    \ge -\frac{2|z_3-z_1|}{1/2-|z_3-z_1|}(M_h(z_3,1/2)-\log|h(z_3)|)
   \ge -C(M_h(0,3/4)-M_h(0,1/4)).
  \end{multline*}
  Plugging these estimates in \cref{eq:B-initial} we get
  \begin{equation}\label{eq:Bern-h-P}
    B_\phi(\mu;z_0,r)\le Cr(M_h(0,3/4)-M_h(0,1/4))+B_P(\mu;z_0,r).
  \end{equation}
  Recall that we know $ B_P(\mu;z_0,r)\le -n\log\mu $, so to get the conclusion we just have to estimate
  $ M_h(0,3/4)-M_h(0,1/4) $. Given $ z\in \cD(0,3/4) $,  apply
  the submean value property to get
  \begin{equation*}
    \log|h(z)|\le\frac{1}{2\pi}\int_0^{2\pi} \log|\phi(z+e^{i\theta}/4)|\,d\theta-
    \frac{1}{2\pi}\int_0^{2\pi} \log|P(z+e^{i\theta}/4)|\,d\theta\le n\log 4
  \end{equation*}
  and conclude $ M_h(0,3/4)\le n\log 4 $. We used the assumption that $ M_{\phi}(0,1)\le 0 $ and
  the fact that
  \begin{equation}\label{eq:log-int}
    \frac{1}{2\pi}\int_0^{2\pi} \log|z-a_k+e^{i\theta}/4|\,d\theta
    = \begin{cases}
      \log \frac{1}{4} &, |z-a_k|\le \frac{1}{4}\\
      \log|z-a_k| &, |z-a_k|>\frac{1}{4}
    \end{cases}
    \ge \log \frac{1}{4}.
  \end{equation}
  Since clearly $ M_P(0,1/4)\le 0 $, we  have $ M_h(0,1/4)\ge M_\phi(1/4) $. So,
  \begin{equation*}
    M_h(0,3/4)-M_h(0,1/4)\le C(n-m)
  \end{equation*}
  and the conclusion follows.
\end{proof}

\begin{remark}
  (1) It is not true that conclusion of Proposition~\ref{prop:bernstein-analytic}
  can be made just in terms of the number $n$ of zeros of $\phi$. Some estimate for
  $ M_\phi(0,1/4)$ is really needed. Here is an elementary example $\phi(z)=\exp(-N+Nz)$,
  $z\in\cD(0,1)$ and $N>0$ is arbitrary. Clearly,  $M_\phi(0,1)=0$, $n=0$. On the other hand
  $ M_\phi(0,1/4)\simeq -N$, $B_\phi(1/4;0,1/8)\simeq N$.

  (2) It is known from \cite{RoyYom97} that if we have control on the valency of the function $ \phi $,
  instead of just the number of zeros, then the estimate for $ M_\phi(0,1/4) $ is not needed anymore.
\end{remark}

\section{Weierstrass' Preparation Theorem and Bernstein Exponents}

We start with a statement of the classical Weierstrass' preparation theorem attuned to our purposes.
\begin{lemma}\label{lem:weier}
  Let $f(z,w)$ be analytic function
  on a polydisk
  \begin{equation*}
    \cP:=\cD(z_0,R_0)\times \cD(w_{0},R_0)\subset \mathbb{C}^2,\quad \ R_0>0.
  \end{equation*}
  Assume that $f(\cdot, w)$ has no zeros on some circle $\Gamma_{\rho_0}=\{ z:|z-z_0|=\rho_0 \}$,
  $ 0<\rho_0<R_0/2 $,
  for any $w \in  \cD(w_0, r_1)$, $0<r_1<R_0$.  Then there exists a Weierstrass polynomial $P(z,w)
  = z^k +a_{k-1} (w) z^{k-1} + \cdots + a_0 (w)$ with $a_j(w)$ analytic in
  $\cD(w_{0}, r_1)$ and an analytic function $g(z,w), (z,w) \in \cP':=\cD(z_0, \rho_0) \times \cD(w_{0}, r_1)$
  so that the following properties hold:
  \begin{enumerate}
  \item[(a)] $f(z,w) = P(z,w) g(z,w)$ for any $(z,w) \in \cP'$.

  \item[(b)] $g(z,w) \ne 0$ for any $(z,w) \in \cP'$.

  \item[(c)] For any $w \in \cD(w_0,r_1)$, $P(\cdot,w)$ has no zeros in $\C \setminus \cD(z_0,\rho_0)$.

  \item[(d)] We have
    \begin{gather}
      \label{eq:g-lb}\left( \inf_{\Gamma_{\rho_0}\times \cD(w_0,r_1)}\log|f| \right)-k\log(2\rho_0)
      \le \inf_{\cP'}\log|g|,\\
      \label{eq:g-ub}\sup_{\cP'}\log|g|\le\left(\sup_{\cP} \log |f|\right)+k\log \frac{2}{R_0}.
    \end{gather}
  \end{enumerate}
\end{lemma}
\begin{proof}
  By the usual Weierstrass argument, one notes that
  $$
  b_p(w) := \sum^k_{j=1} \zeta^p_j (w) = \frac{1}{2\pi i} \oint\limits_{\Gamma} z^p \
  \frac{\partial _z f(z, w)}{f(z, w)}\, dz
  $$
  are analytic in $\cD(w_0,r_1)$.  Here $\zeta_j (w)$ are the zeros of $f(\cdot, w)$
  in $\cD(z_0,\rho_0)$.  Since the coefficients $a_j(w)$ are
  linear combinations of the $b_p$,
  they are analytic in $w$.  Analyticity
  of $g$ follows by standard arguments. We just have to prove (d). Since all the roots of $ P(\cdot,w) $ are in
  $ \cD(z_0,\rho_0) $, we have $ \sup_{\cP'}|P|\le (2\rho_0)^k $ and \cref{eq:g-lb} follows using the
  minimum modulus principle. 
  Note that actually the function $ g $ can be defined on $ \cP $ as $ g=f/P $ and it is analytic there.
  Given $ (z,w)\in \cP' $,  apply
  the sub-mean value property for subharmonic functions to get
  \begin{multline*}
    \log|g(z,w)|\le \frac{1}{2\pi}\int_0^{2\pi} \log|f(z+R_0e^{i\theta}/2,w)|\,d\theta-
    \frac{1}{2\pi}\int_0^{2\pi} \log|P(z+R_0e^{i\theta}/2,w)|\,d\theta\\
    \le \left(\sup_{\cP} \log |f|\right)+k\log \frac{2}{R_0}.
  \end{multline*}
  The estimate on the mean value of the polynomial follows by considerations analogous to \cref{eq:log-int}.
\end{proof}

Next we describe how Bernstein exponents rule the application of
 Lemma~\ref{lem:weier}.

\begin{lemma}\label{lem:6.3}
  Let $f$ be analytic on  $B(0,1)$,
  $ M\ge \sup_{B(0,1)}\log|f| $, $ \um= M-B $, $ B\gg 1 $, $ \fe_1 $ a  $ \um $-regular direction for $ f $ at $ 0 $
  (recall   \cref{defi:regulardir}),
  and $ \fe_2 $ another non-collinear direction.
  With a slight abuse of notation we denote by $f(z,w)$ the function in the new coordinates with respect to the basis
  $\mathfrak{e}_1,\mathfrak{e}_2$.
  Then there exists a circle $\Gamma_{\rho_0} = \left\{|z| = \rho_0\right\}$,
  $1/8 < \rho_0 < 1/4$, and $r_1=\exp \left(-CB\right)$, with $ C>1 $ an absolute constant, such that
  \begin{equation}\label{eq:f-blb}
    \inf_{\Gamma_{\rho_0}\times \cD(0,r_1)} \log|f|\ge \exp(M-CB).
  \end{equation}
  In particular, Lemma~\ref{lem:weier} applies for $f(z,w)$ with this choice of $\rho_0$ and~$r_1$,
  as well as with $k\lesssim B$ and $ \delta\ge M-CB $.

\end{lemma}
\begin{proof}
  Since $ \fe_1 $ is a $ \um $-regular direction, there exists $ z_1 $, $ |z_1|= 1/4 $,
  such that $ \log|f(z_1,0)|\ge \um $.
  Due to Cartan's estimate one has
  \begin{equation}\label{eq:fz0-Cartan}
      \log|f(z,0)| \ge M - C(M-\um)= M-CB
  \end{equation}
  for any $z \in \cD(0, 1/4) \setminus \cB$, where $\cB \in \car_1\left(C', C'B\right)$, $ C'\gg 1 $.
  As a consequence of the definition of $ \Car_1 $ sets, we can choose $1/8 < \rho_0 < 1/4$ such that
  $\cB \cap \Gamma_{\rho_0} =\emptyset$.
  Then
  \begin{equation}\label{eq:w0}
      |f(z, 0)| \ge \exp \left(M - CB \right)
  \end{equation}
  for any $z \in \Gamma_{\rho_0}$.  Note that due to Cauchy's estimates
  \begin{equation*}
    |f(z,w) - f(z,0)| \lesssim e^M |w|
  \end{equation*}
  for any $z \in \cD(0, 1/2)$, $w \in \cD(0, 1/2)$.  Taking into
  account~\eqref{eq:w0},  one obtains
  \begin{equation*}
      |f(z,w)| >\exp \left(M - CB \right)
  \end{equation*}
  for any $z \in \Gamma_{\rho_0}$, provided $w \in\cD(0, r_1)$, $r_1 =\exp \left(-CB\right)$, with $ C $ large
  enough (of course, $ C $ is larger than in \cref{eq:w0}). This proves \cref{eq:f-blb} and allows us to apply
  \cref{lem:weier} as stated. For the bound on the degree of the Weierstrass polynomial note that
  by Jensen's formula applied to $f(z,0)$, $z\in\cD(z_1,1/2)$,
  \begin{equation*}
    k \le \#\left\{z \in \cD(0,1/4): f(z,0) = 0 \right\}
    \le \#\left\{z \in \cD(z_1, 1/2): f(z,0) = 0 \right\} \lesssim B.
  \end{equation*}
\end{proof}

\begin{remark}\label{rem:Weier-simultaneous}
 $(1)$ Due to \cref{lem:6.3c}, we will always apply the previous lemma with $ B\simeq B_f(1/4;0,1) $. This is how the
  Bernstein exponent determines the size of the polydisk on which we have the Weierstrass factorization.

  $(2)$ If we are given two functions $ f_1,f_2 $ satisfying the assumptions of \cref{lem:6.3} with the same $ M $ and
  $ B $, then it is clear from the proof of the lemma that we can arrange for the conclusion to hold for both
  functions with the same choice of $ \rho_0 $ and $ r_1 $. Indeed, one only needs to choose $ \rho_0 $ such that
  $ \Gamma_{\rho_0}\cap(\cB_1\cup \cB_2)=\emptyset $, where $ \cB_i $ are the Cartan sets needed to guarantee
  \cref{eq:fz0-Cartan} for $ f_i $.
\end{remark}

\section{Resultants}

We briefly recall the definition of the resultant of two univariate polynomials and some of the basic properties
that we'll use.
Let $f(z) = a_nz^n + a_{n-1} z^{n-1} + \cdots + a_0$,
$g(z) = b_mz^m + b_{m-1} z^{m-1} + \cdots + b_0$ be polynomials,
$a_i, b_j \in \IC$, $a_n\neq 0$, $b_m\neq 0$.  Let $\zeta_i$, $1 \le i \le n$ and $\eta_j$,
$1 \le j \le m$ be the zeros of $f(z)$ and $g(z)$ respectively.  The
resultant of $f$ and $g$ is defined as follows:
\begin{equation}\label{eq:resdef}
  \Res(f, g) = a_n^mb_m^n\prod_{i,j} (\zeta_i - \eta_j)=(-1)^{mn}b_m^n\prod_{j} f(\eta_j)
  =(-1)^{mn}a_n^m\prod_{i} g(\zeta_i).
\end{equation}
The resultant $\Res(f, g)$ can be expressed explicitly in terms of the
coefficients (see \cite{Lan02}):
\begin{equation}\label{eq:resdef1}
  \Res(f, g) = \left|\begin{array}{ll}
                       {\overbrace{\begin{array}{lll}
                                     a_n & 0 & \cdots\\
                                     a_{n-1}  & a_n & \cdots\\
                                     a_{n-2} & a_{n-1} & \cdots\\
                                     \cdots & \cdots & \cdots\\
                                     \cdots & \cdots & \cdots\\
                                     \cdots & \cdots & \cdots\\
                                     a_0 & a_1 &  \\
                                     0 & a_0 &  \end{array}}^m} &
                                                                  {\overbrace{\begin{array}{llll}
                                                                                b_m & 0 & \cdots & 0\\
                                                                                b_{m-1} & b_m & \cdots & \cdots\\
                                                                                b_{m-2} & b_{m-1} &\cdots & \cdots\\
                                                                                \cdots & \cdots & \cdots & \cdots\\
                                                                                \cdots & \cdots & \cdots & \cdots\\
                                                                                \cdots & \cdots & \cdots& \cdots\\
                                                                                  &&&\\
                                                                                  &&&\\
                                                                              \end{array}}^n}
                     \end{array}\right|
                 \end{equation}

\begin{lemma}\label{lem:resbasicprop1}
  Let $ f,g,\zeta_i,\eta_j $ as above. Set
  \[
   t_f=\min (|a_n|,1),\quad t_g=\min (|b_m|,1),\quad T_f=
    \max_i(\max |a_i|,1),\quad T_g=
    \max_j(\max |b_j|,1),
    \]
    \[
    R_f=t_f^{-1}T_f m,\quad R_g=t_g^{-1}T_g n.
  \]
  The following statements hold.

  $ (0) $    $ \max |\zeta_i|\le R_f,\quad \max |\eta_j|\le R_g $.

  $ (1) $
  If
  \[
    \big|\Res(f, g)\big|<\delta^{m}t_g^{n},\quad 0<\delta<1,
  \]
  then there exists $j$ such that
  \[
    \big|f(\eta_j)\big|< \delta.
  \]
  In particular, there exists
  $|z|\le R_g$ such that $ \max ( \big|f(z)\big|,\big|g(z)\big|)< \delta $.

  $ (2) $
  If there exists $z$ such that with $s=\max (m,n)$, $t=\min (t_f,t_g)$ holds
  \[
    \max [ \big|f(z)\big|,\big|g(z)\big|]<t\delta^{s},\quad 0<\delta<1,
  \]
  then
  \[
    \big|\Res(f, g)\big|< t^{2s}(2R)^{s^2}\delta,
  \]
  $R=\max (R_f,R_g)$.
\end{lemma}
\begin{proof}
  (0) follows by noting that, for example,
  \begin{equation*}
    |a_n||\zeta_i|^n\le (\max |a_i|)(|\zeta_i|^{n-1}+\dots+|\zeta_i|+1)\le (\max |a_i|) n\max(|\zeta_i|^{n-1},1).
  \end{equation*}
  (1) follows by contradiction from \cref{eq:resdef}. For (2) note that there must exist
  $ \zeta_{i_0}, \eta_{j_0} $ such that $ |z-\zeta_{i_0}|<\delta $, $ |z-\eta_{j_0}|<\delta $ and therefore,
  using (0) and \cref{eq:resdef},
  \begin{equation*}
    |\Res(f,g)|\le t^{2s}(2R)^{s^2-1}|\zeta_{i_0}-\eta_{j_0}|<t^{2s}(2R)^{s^2}\delta.
  \end{equation*}
\end{proof}

\section{Refinement of the Assumption \cref{eq:cartanmaincond}}\label{secCar2sets}

We give a simple argument showing that by making some small adjustments we actually have $ K_1\le H_0^C $ in
\cref{eq:cartanmaincond}.

\begin{lemma}\label{lem:K1-sa-bound}
  Using the notation and assumptions of \cref{thm:C} we have that
  \begin{equation*}
    \cN(F,\epsilon_0/2)\cap B(0,1/2)\in \Car_{2,0}(H_1/2,H_0^C)
  \end{equation*}
  where $ C $ is some large absolute constant.
\end{lemma}
\begin{proof}
  Let $ f_{i,N} $ be the degree $ N $ Taylor polynomials (at the origin) associated with $ f_i $, $ i=1,2 $
  (recall that $ F=(f_1,f_2) $). Since $ M_{f_i}(0,1)\le 0 $, a standard application of the Cauchy estimates
  yields that
  \begin{equation*}
    |f_i-f_{i,N}|<\epsilon_0/100
  \end{equation*}
  for $ N=C\log \epsilon_0^{-1}=C H_0 $, $ C\gg 1 $. Let $ F_N=(f_{1,N},f_{2,N}) $. We have
  \begin{multline}\label{eq:sa-approx}
    \cN(F,\epsilon_0/2)\cap B(0,1/2)\subset \cN(F_N,3\epsilon_0/4)\cap B(0,1/2)\\
    \subset \cN(F,\epsilon_0)\cap B(0,1/2).
  \end{multline}
  The set $ \cN(F_N,3\epsilon_0/4)\cap B(0,1/2) $ is semialgebraic of degree less than $ CN $ and therefore it
  has at most $ N^C $ connected components. We refer to \cite[Ch.~9]{Bou05} for a brief review of semialgebraic
  sets and their properties. It follows from our assumptions that $ \cN(F_N,3\epsilon_0/4)\cap B(0,1/2) $ is
  covered by less than $ K_1 $ balls of radius $ \exp(-H_1) $. Therefore, each connected component of
  $ \cN(F_N,3\epsilon_0/4)\cap B(0,1/2) $ can be covered by just one ball of radius smaller than
  \begin{equation*}
    CK_1 \exp(-H_1)\le \exp(-H_1/2)
  \end{equation*}
  (recall that $ \log K_1\ll H_1 $) and so
  \begin{equation*}
    \cN(F_N,3\epsilon_0/4)\cap B(0,1/2)\in \Car_{2,0}(H_1/2,N^C).
  \end{equation*}
  The conclusion now follows from \cref{eq:sa-approx}.
\end{proof}

\section{Proofs of Theorems A,B,C}\label{sec:thmABC}

We start with the proof of \cref{thm:C}.

\begin{proof}[Proof of \cref{thm:C}]
  Take $ \uv_0=(z_0,w_0)\in B(0,1/4) $. By our assumptions
  \begin{equation*}
    B_{f_i}(1/4;\uv_0,1/4),\quad M_{f_i}(\uv_0,1/4)\le 0.
  \end{equation*}
  Due to \cref{lem:6.3c}, we can find unit vectors $ \fe_1,\fe_2$, $ |\langle \fe_1,\fe_2 \rangle|\ll 1 $,
  that are $ \um $-regular at $ \uv_0 $ for both $ f_1,f_2 $ restricted to $ B(\uv_0,1/4) $, with $ \um=-CB_0 $,
  $ C\gg 1 $. Then Lemma~\ref{lem:6.3} applies to both $f_1,f_2$ and to both directions $ \fe_1,\fe_2 $. As in
  Lemma~\ref{lem:6.3}, with a slight abuse of notation we denote by $ f_i(z,w) $ the functions in the coordinates
  with respect to the basis $ \fe_1,\fe_2 $ centered at $ \uv_0 $ and with the obvious rescaling needed to apply
  the lemma.
  Applying Lemma~\ref{lem:6.3} (see also \cref{rem:Weier-simultaneous})
  in the direction of $ \fe_1 $ (and $ \fe_2 $ as the choice of non-collinear direction)
  we can write
  \begin{align*}
    f_i(z,w) &= P_i(z,w) g_i(z,w),\\
    P_i(z,w)
             & = z^{k_i} +a_{i,k_i-1} (w) z^{k_i-1} + \cdots + a_0 (w)
  \end{align*}
  on $ \cP:=\cD(0,\rho_{0})\times \cD(0,r_1) $, $ 1/8<\rho_{0}<1/4 $, $ r_1=\exp(-CB_0) $, where
  the coefficients $ a_{i,j}(w) $ are analytic
  on $ \cD(0,r_1) $, $ g_i $ are analytic and non-vanishing on $ \cP $, the polynomials $ P_i(\cdot,w) $,
  $ w\in \cD(0,r_1) $, have no  zeroes in $ \C\setminus \cD(0,\rho_0) $, and $ k_i\lesssim B_0$. Furthermore,
  using part (d) of \cref{lem:weier},
  \begin{equation}\label{eq:g-lbub}
    -B_0\lesssim    \inf_{\cP}  \log \bigl | g_i \bigr|
    \le \sup_{\cP}\log \bigl | g_i\bigr|\lesssim B_0.
  \end{equation}
  Let
  \[
    R(w)=\Res \left(P_1(\cdot,w),P_2(\cdot,w)\right).
  \]
  Note that by \cref{eq:resdef1}, $ R $ is analytic on $ \cD(0,r_1) $.
  Since we chose $ \uv_0\in B(0,1/4) $, the polydisk $ \cP $ is a subset of
  $ B(0,1/2) $, as a set in the standard coordinates. This allows to use the hypothesis to guarantee that
  there exist points $ \uv_j=(z_j,w_j) $ (expressed in the $ \fe_1,\fe_2 $ coordinates), $ 1\le j \le J\le K_1 $
  such that for
  \begin{equation*}
    (z,w)\in \cP\setminus \left( \bigcup_{j=1}^J B(\uv_j,C\exp(-H_1))\right)
  \end{equation*}
  we have
  \begin{equation*}
    \max(|f_1(z,w)|, |f_2(z,w)|)\ge \exp(-H_0)/\sqrt{2}
  \end{equation*}
  and by \cref{eq:g-lbub}
  \begin{equation}\label{eq:P-lb}
    \max(|P_1(z,w)|,|P_2(z,w)|)\gtrsim \exp(-H_0-CB_0).
  \end{equation}
  Note that we used the radius $ C\exp(-H_1) $ instead of $ \exp(-H_1) $ to account for the distortion under the
  change of coordinates. Since we are assuming that $ H_1\gg B_0 $ and $ \log K_1\ll H_1 $, we can find
  $$ w\in \cD(0,r_1/4)\setminus \bigcup_{j=1}^J \cD(w_j,C\exp(-H_1)).$$ For any such $ w $ \cref{eq:P-lb} holds for
  any $ z\in \cD(0,\rho_0) $ and by part (1) of \cref{lem:resbasicprop1}
  \begin{equation*}
    \log|R(w)|\gtrsim -B_0H_0-B_0^2\gtrsim -B_0H_0.
  \end{equation*}
  Note that by the definition of the resultant \cref{eq:resdef}, we have $ \sup_{\cD(0,r_1)}|R(w)|\le 1 $.
  Take $ H\gg 1 $.
  Applying
  Cartan's estimate, we get
  \begin{equation*}
    \log|R(w)|\gtrsim -HB_0H_0
  \end{equation*}
  for any $$ w\in \cD(0,r_1/4)\setminus \cB,\quad  \cB=\bigcup_{1\le k\le K} \cD(w_k',r_1\exp(-H)),\quad
  K\lesssim B_0 H_0. $$ By part (2) of \cref{lem:resbasicprop1},
  \begin{equation*}
    \max(|P_1(z,w)|,|P_2(z,w)|)\ge \exp(-CHB_0^2H_0)
  \end{equation*}
  for any $ w\in \cD(0,r_1/4)\setminus \cB $ and $ z\in \C $. Using \cref{eq:g-lbub}, we get
  \begin{equation}\label{eq:F-lb}
    |F(z,w)|\gtrsim \exp(-CHB_0^2H_0-CB_0)\ge \exp(-HB_0^2H_0)
  \end{equation}
  for any $ w\in \cD(0,r_1/4)\setminus \cB $ and $ z\in \cD(0,\rho_0) $.
  Applying Lemma~\ref{lem:6.3} again
  in the direction of $ \fe_2 $ (and with $ \fe_1 $ as the choice of non-collinear direction) and repeating the
  above argument we get that there exist  $ 1/8<\tilde \rho_0<1/4 $, $ \tilde r_1=\exp(-CB_0) $, such that
  \cref{eq:F-lb} also holds for any $ z\in \cD(0,\tilde r_1/4)\setminus \tilde \cB $
  and $$ w\in \cD(0,\tilde \rho_0),\quad  \tilde \cB=\bigcup_{1\le \ell \le L} \cD(z_\ell',\tilde r_1\exp(-H)),\quad
  L\lesssim B_0H_0, $$ In particular, \cref{eq:F-lb} holds for any
  \begin{equation*}
    (z,w)\in \cD(0,\tilde r_1/4)\times \cD(0, r_1/4)
    \setminus \left( \bigcup_{k,\ell} \cD(z_\ell',\tilde r_1\exp(-H))\times\cD(w_k',r_1\exp(-H))\right).
  \end{equation*}
  Going back to standard coordinates we obtained that there exist less than $ CB_0^2H_0^2 $
  points $ \uv_j' $ such that
  \cref{eq:F-lb} holds for any
  \begin{equation*}
    (z,w)\in B(\uv_0,\exp(-CB_0))\setminus \left( \bigcup_{j} B(\uv_j',\exp(-H)) \right).
  \end{equation*}
  Since \cref{eq:F-lb} holds outside the initial $ \Car_{2,0} $ set, we only need to apply the above argument
  on $ K_1 $ balls covering the initial set to get the conclusion.
\end{proof}

We will need the following lemma for the proof of \cref{thm:A}.

\begin{lemma}\label{lem:complexzerosetextends}
  Let $f$ be analytic on the ball
  $B(0,1)$, $\cZ=\{\uv \in B(0,1): f(\uv)=0\}$. Let $\cB\in \car_{2,0}(H_1,K_1)$, $H\gg 1$, $ \log K\ll H $.
  If $ 0\in \cZ $, then
  \[
    B(0,1/4)\cap \cZ \setminus \cB\neq \emptyset.
  \]
\end{lemma}
\begin{proof}
  We argue by contradiction. Assume $ B(0,1/4)\cap \cZ \subset \cB $. By the assumptions on $ \cB $ we can find
  $ 1/8<r<1/4 $ such that $ \cB\cap B(0,r) $ is compactly contained in $ \cB $. Therefore the zero set of $ f $
  restricted to $ B(0,r) $ is compactly contained in $ B(0,r) $ and $ \cZ\cap B(0,r) $ is a compact analytic
  variety in $ \C^2 $. This cannot be, because compact analytic varieties in $ \C^2 $ are necessarily
  finite sets (see for example \cite{Chi89}
  ) and analytic functions of several variables cannot have
  isolated zeros (recall that $ 0\in \cZ $).
\end{proof}

\begin{proof}[Proof of \cref{thm:A}]
  (1)
  Take $ \uv_0\in B(0,1/8)\cap \cZ $, $ \cZ=\{ f_2=0 \} $, $ 0<R\le 1/4 $. Let $ H=C\max(\log R^{-1},C_0) $
  with $ C $ large enough (recall that $ C_0=\log (K_1 B_0^2 H_0^2) $). By \cref{thm:C} we have
  \begin{equation*}
    |F(\uv)|\ge \exp(-HB_0^2H_0)
  \end{equation*}
  for all $ \uv\in B(0,1/4)\setminus \cB $, $ \cB\in \Car_{2,0}(H,K) $, $ K\lesssim K_1B_0^2H_0^2 $.
  Note that $ B(\uv_0,R)\subset B(0,1/4) $ and our choice of $ H $ is such that we can apply
  \cref{lem:complexzerosetextends} to $ f_2 $ restricted to $ B(\uv_0,R) $ and the above $ \cB $ (after an obvious
  rescaling). So,
  there exists $ \uv_1\in \cB(\uv_0,R/4)\cap \cZ\setminus \cB $. Note that we have
  \begin{equation*}
    |f_1(\uv_1)|=|F(\uv_1)|\ge \exp(-HB_0^2H_0)
  \end{equation*}
  and therefore
  \begin{equation*}
    M_{f_1}(\cZ,\uv_0,R/4)\ge -HB_0^2 H_0.
  \end{equation*}
  The first statement now follows by recalling that
  \begin{equation*}
    M_{f_1}(\cZ,\uv_0,R)\le M_{f_1}(0,1)\le 0.
  \end{equation*}

  (2) Take $ \uv_0\in B(0,1/8) $. By our assumptions
  \begin{equation*}
    B_{f_2}(1/4;\uv_0,1/4)\le B_0,\quad M_{f_2}(\uv_0,1/4)\le 0.
  \end{equation*}
  Due to \cref{lem:6.3c}, we can find a unit vector $ \fe_1$,
  that is $ \um $-regular at $ \uv_0 $ for $ f_2 $ restricted to $ B(\uv_0,1/4) $, with $ \um=-CB_0 $,
  $ C\gg 1 $. Let $ \fe_2 $ be another unit vector orthogonal to $ \fe_1 $.
  As in
  Lemma~\ref{lem:6.3}, with a slight abuse of notation we denote by $ f_i(z,w) $ the functions in the coordinates
  with respect to the basis $ \fe_1,\fe_2 $ centered at $ \uv_0 $ and with the obvious rescaling needed to apply
  the lemma.   Applying Lemma~\ref{lem:6.3}
  in the direction of $ \fe_1 $ (with $ \fe_2 $ as the choice of non-collinear direction)
  we can write
  \begin{align*}
    f_2(z,w) &= P(z,w) g(z,w),\\
    P(z,w)
             & = z^{k} +a_{k-1} (w) z^{k-1} + \cdots + a_0 (w)
  \end{align*}
  on $ \cP:=\cD(0,\rho_{0})\times \cD(0,r_1) $, $ 1/8<\rho_{0}<1/4 $, $ r_1=\exp(-CB_0) $, where
  the coefficients $ a_{j}(w) $ are analytic
  on $ \cD(0,r_1) $ and $ k\lesssim B_0 $. Since we also have that  $ g $ is analytic and non-vanishing on $ \cP $,
  \begin{equation*}
    \cZ\cap \cP=\cZ_P\cap \cP,\quad \cZ_P:= \{ (z,w)\in \C\times \cD(0,r_1) : P(z,w)=0 \}.
  \end{equation*}
  It is well known (see \cite{Chi89}) that for any point $ (z,w) $ of the variety
  $ \cZ_P $, there exist  $ \epsilon>0 $, $ \delta>0 $ such that the following statements hold.
  \begin{enumerate}[leftmargin=1.5em,label=(\roman*)]
  \item
    If $ (z,w) $ is a regular point, then there exists an analytic function
    $ \zeta: \cD(0,\epsilon)\to \cD(0,\delta) $
    such that
    \begin{equation*}
      \cZ_P\cap \left( \cD(z,\delta)\times \cD(w,\epsilon) \right)= \{ (z+\zeta(w'-w),w') : w'\in \cD(w,\epsilon) \}.
    \end{equation*}
  \item
    If $ (z,w) $ is a singular point, then there exist integers $ p_i\ge 1 $ and analytic functions
    $ \zeta_i: \cD(0,\epsilon)\to \cD(0,\delta) $, $ 1\le i\le i_0(z,w)\le k $, such that $ \sum_i p_i\le k $ and
    \begin{equation*}
      \cZ_P\cap \left( \cD(z,\delta)\times \cD(w,\epsilon) \right)
      = \bigcup_i\{ (z+\zeta_i((w'-w)^{\frac{1}{p_i}}),w') : w'\in \cD(w,\epsilon) \}.
    \end{equation*}
  \end{enumerate}

  By compactness we can cover $ B(0,1/8)\cap \cZ $ by finitely many polydisks  $ \frac{1}{2}\cP $ (more precisely,
  by their preimages under the change of variables we assumed above) and in turn
  $ \cZ\cap \frac{1}{2}\cP $ can be covered by finitely many polydisks
  $ \cD(z_j,\delta_j)\times \cD(w_j,\epsilon_j/8)$ with $ (z_j,w_j)\in \cZ_P $ and $ \epsilon_j,\delta_j $ as above.
  We will also use $ \zeta_j $ and  $ \zeta_{i,j} $ the functions associated with $ (z_j,w_j) $.
  Let $ r_0>0 $ be the minimum over all the $ \epsilon_j $ needed to cover $ B(0,1/8)\cap \cZ $. Near each
  $ (z_j,w_j) $ we will define
  local charts and show we can control the Bernstein exponent of $ f_1 $ in the local charts. The control over the
  Bernstein exponent will follow from \cref{thm:C} and \cref{prop:bernstein-analytic}. To this end we take
  $ H=C(\log r_0^{-1}) C_0 $, with $ C\gg 1 $ large enough and we note that, with this choice of $ H $,
  \cref{thm:C} guarantees that
  \begin{equation}\label{eq:F-lb-2}
    |F(z,w)|\ge \exp(-HB_0^2H_0),\ \forall (z,w)\in \cP\setminus (\C\times \cB)
  \end{equation}
  where $ \cB $ is a union of disks with the sum of the radii much smaller than $ r_0^k $ (recall that
  $ k\lesssim B_0\ll H_0 $).
  To define the charts we distinguish two cases.

  \noindent (i) $ (z_j,w_j) $ is regular. Let
  \begin{equation*}
    \psi_j(w)=(z_j+\zeta_j(w),w_j+w),\  w\in \cD(0,\epsilon_j).
  \end{equation*}
  It follows from \cref{eq:F-lb-2} that
  \begin{equation*}
    M_{f_1\circ \psi_j}(0,1/4)\ge -HB_0^2H_0.
  \end{equation*}
  By \cref{prop:bernstein-analytic} (recall that $ M_{f_1}(0,1)\le 0 $) it is clear that
  \begin{equation*}
    B_{f_1\circ \psi_j}(1/4;z,r)\lesssim H\le C(f_2)C_0 B_0^2H_0
  \end{equation*}
  when $ \cD(z,r)\subset \cD(0,\epsilon_j/8) $. This shows the conclusion of part (2) holds if we define the local
  chart by rescaling $ \psi_j\vert_{\cD(0,\epsilon_j/8)} $.

  \noindent (ii) $ (z_j,w_j) $ is singular. Let
  \begin{equation*}
    \psi_{i,j}(w)=(z_j+\zeta_{i,j}(w),w_j+w^{p_i}),\ w\in \cD(0,\epsilon_j).
  \end{equation*}
  It follows from \cref{eq:F-lb-2} that
  \begin{equation*}
    M_{f_1\circ \psi_{i,j}}(0,1/4)\ge -HB_0^2H_0
  \end{equation*}
  (recall that $ p_i\le k $) and therefore
  \cref{prop:bernstein-analytic} guarantees that
  \begin{equation*}
    B_{f_1\circ \psi_{i,j}}(1/4;z,r)\lesssim HB_0^2H_0\le C(f_2)C_0 B_0^2H_0
  \end{equation*}
  when $ \cD(z,r)\subset \cD(0,\epsilon_j/8) $. This shows that the conclusion holds if corresponding to each
  $ w\in \cD(0,\epsilon_j/8)\setminus \{ 0 \} $ we define a local chart  by rescaling
  $ \psi_{i,j}\vert_{\cD(w,r)} $, where $ \cD(w,r) $ is the largest disk about $ w $ in $ \cD(0,\epsilon_j/8) $ on
  which $ w^{p_i} $ is one-to-one.

  Clearly the above charts cover $ \reg \cZ \cap B(0,1/8) $ and we can complete an atlas of $ \reg \cZ $ by adding
  charts whose ranges don't intersect $ B(0,1/8) $. This concludes the proof.
\end{proof}

Next we give an example showing that the $ \log R^{-1} $ is actually necessary in part (1) of \cref{thm:A}.

\begin{example}\label{ex:logR}
  Let
  \begin{equation*}
    f_1(z,w)=z^2+w,\quad f_2(z,w)=zw,\quad \cZ= \{ f_2=0 \}.
  \end{equation*}
  Let $ R\ll 1 $, $ \uv_0=(R/4,0) $. Then straightforward computations show that
  \begin{equation*}
    \sup_{\cB(\uv_0,R/4)\cap \cZ}\log|f_1(z,w)|=\sup_{|z-R/4|<R/4}\log|z^2|=\log \left( \frac{R}{2} \right)^2,
  \end{equation*}
  and
  \begin{multline*}
    \sup_{\cB(\uv_0,R)\cap \cZ}\log|f_1(z,w)|
    =\max \left( \sup_{|z-R/4|<R}\log|z^2| ,\sup_{|w|^2+(R/4)^2<R^2}\log|w| \right)\\
    =\max \left( \log \left( \frac{5R}{4} \right)^2,\log \frac{\sqrt{15}R}{4} \right)
    = \log \frac{\sqrt{15}R}{4},
  \end{multline*}
  provided $ R $ is small enough ($ R<1/2 $ is enough). Therefore,
  \begin{equation*}
    B_{f_1}(1/4;\cZ,\uv_0,R)= C+\log R^{-1}.
  \end{equation*}
\end{example}

Finally, we will prove \cref{thm:B}, but we first establish an auxilliary result.
To this end we will need the following extension of the classical  B\'ezout theorem.
Suppose we have a system of $ n $ complex polynomial equations $ f_i(z_1,...,z_n)=0 $, $i=1,..,n$.
Let $ \cZ_1,\ldots,\cZ_s $ be the irreducible components of the variety defined by the system.
Then
\begin{equation}\label{eq:Bezout}
  \deg(\cZ_1)+\dots+\deg(\cZ_s)\le \deg f_1\times \dots \times \deg f_n.
\end{equation}
The authors are grateful to J\'anos Koll\'ar  and Mihnea Popa for pointing out this version of the B\'ezout
bound (for a more general result see \cite[Thm.~12.3]{Ful98}).

\begin{lemma}\label{lem:bezoutpuiseoux1}
  Let $f(z,w)$, $g(z,w)$ be non-constant polynomials with no common factors.
  Let $\zeta(w)$ be an analytic function on $ \cD(0,r_0)$
  such that
  \begin{equation}\label{eq:regularity-hyp}
    \{ (\zeta (w),w) : w\in \cD(0,r_0) \}\subset \reg \{ f(z,w)=0 \}.
  \end{equation}
  Then there exists at most one straight line $\cL\subset \C$ through
  the origin such that
  \begin{equation}\label{eq:bezoutpuiseux2}
    \#\{\xi\in (-r_0,r_0):g(\zeta(\xi ),\xi )\in \cL\}> (\deg f)^2\deg g.
  \end{equation}
\end{lemma}
\begin{proof}
  Let $ \cL $ be a line through the origin. We first argue that if \cref{eq:bezoutpuiseux2} holds, then we must
  have $ \{ g(\zeta (\xi),\xi): \xi\in (-r_0,r_0) \}\subset \cL $.
  Write
  \begin{equation*}
    f(z,\xi):=P(x+iy,\xi)+iQ(x+iy,\xi)=\hat P(x,y,\xi)+i\hat Q(x,y,\xi)
  \end{equation*}
  where $P,Q$ are the real and imaginary parts of $f$, and
  $\hat P,\hat Q$ are real polynomials of three real variables $x,y,\xi$.
  Clearly $ \deg \hat P=\deg \hat Q= \deg f $.
  Similarly write
  \begin{equation*}
    g(z,\xi):=U(x+iy,\xi)+iV(x+iy,\xi)=\hat U(x,y,\xi)+i\hat V(x,y,\xi).
  \end{equation*}
  Without loss of generality we may assume that the line
  $ \cL $ is horizontal. If \cref{eq:bezoutpuiseux2} holds,
  then the system
  \begin{equation}\label{eq:system}
    \hat P=0,\quad \hat Q=0,\quad \hat V=0
  \end{equation}
  has more than $\deg \hat P\times\deg \hat Q \times \deg \hat V$ solutions $\uv_j=(x_j,y_j,\xi_j)$ with
  \begin{equation*}
    \xi_j\in (-r_0,r_0),\quad x_j+iy_j=\zeta (\xi_j),\quad \xi_{j_1}\neq \xi_{j_2}.
  \end{equation*}
  Complexify the variables $x,y,\xi$, and let $ \cZ_1,\ldots,\cZ_s $ be the irreducible components of the complex
  variety defined by the system \cref{eq:system}. By the B\'ezout bound \cref{eq:Bezout}, there exists a component
  $ \cZ_k $ that contains at least two of the solutions $ \uv_j $ and therefore has dimension at least one.
  Let $ \uv_0 $ be one of the solutions contained in $ \cZ_k $.
  We will argue that there exists an analytic mapping
  \begin{equation}\label{eq:mapping}
    t\to \uv(t)=(x(t),y(t),\xi(t))\in \cZ_k,\  t\in \cD(0,\delta)
  \end{equation}
  such that $ \uv(0)=\uv_0 $ and $ \xi(t) $ is non-constant. By \cite{Shi70} we know that
  there exists a neighborhood $ N $ of $ \uv_0 $ in $ \cZ_k $, such that for any
  $ \uv\in N\setminus \{ \uv_0 \} $, there exists a one dimensional irreducible variety $ \cV $  through both
  $ \uv $ and $ \uv_0 $. Since $ \cV $ can be parametrized by a Riemann surface (see \cite[Prop.~6.2]{Chi89})
  we get the existence of a mapping of the form \cref{eq:mapping}. If $ \xi(t) $ is constant, then by the
  uniqueness theorem (see \cite[Cor.~5.3.2]{Chi89}), we must have $ \cV\subset \{ \xi=\xi_0 \} $.
  If this happens for all such mappings obtained by choosing different $ \uv\in N\setminus \{ \uv_0\} $, then
  $ N\subset \{ \xi=\xi_0 \} $, and by the uniqueness theorem, $ \cZ_k\subset \{ \xi=\xi_0 \} $.
  This would contradict
  the fact that $ \cZ_k $ contains two of the solutions $ \uv_j $ (recall that $ \xi_{j_1}\neq \xi_{j_2} $). So we
  proved the existence of the mapping \cref{eq:mapping} with the desired properties. We have
  \begin{equation*}
    f(x(t)+iy(t),\xi(t))=0,\quad V(x(t)+iy(t),\xi(t))=0,\ t\in \cD(0,\delta).
  \end{equation*}
  By the assumption \cref{eq:regularity-hyp}, we get
  \begin{equation*}
    x(t)+iy(t)=\zeta (\xi(t)),
  \end{equation*}
  provided we choose $ \delta $ small enough. Therefore
  \begin{equation*}
    V(\zeta(\xi(t)),\xi(t))=0,\ t\in \cD(0,\delta)
  \end{equation*}
  and since $ \xi(t) $ is non-constant, there exists $ \epsilon>0 $ so that
  \begin{equation*}
    V(\zeta(\xi),\xi)=0,\ \xi\in (\xi_0-\epsilon,\xi_0+\epsilon).
  \end{equation*}
  So $ V(\zeta (\xi),\xi)=0 $ for all $ \xi\in(-r_0,r_0) $, that is
  $ \{ g(\zeta(\xi),\xi): \xi\in (-r_0,r_0) \}\subset \cL $.

  Now we can finish the proof by arguing by contradiction.
  If the conclusion doesn't hold, it follows that we have
  $ \{ g(\zeta(\xi),\xi): \xi\in (-r_0,r_0) \}\subset \cL_1\cap \cL_2= \{ 0 \} $ and
  therefore the system $ f=g=0 $ has infinitely many solutions. By the classical
  B\'ezout theorem, this would contradict the assumption that $ f $ and $ g $
  don't have common factors.
\end{proof}

\begin{proof}[Proof of \cref{thm:B}]
  Let $ \cZ_1,\ldots,\cZ_s $ be the irreducible components of $ \cZ $. Each of them is the zero set of
  an irreducible factor of $ f_2 $. Let $ f_{2,1},\ldots,f_{2,s} $ be such irreducible factors.
  Fix $ k\in \{ 1,\ldots,s \} $ and $ (z_0,w_0)\in \reg \cZ\cap \cZ_k $.
  We can make a change of variables (as in the proof of part (2) of \cref{thm:A}) such that
  $ (z_0,w_0) $ is mapped to the origin and
   we can find an analytic function
  $ \zeta: \cD(0,\epsilon_0)\to \cD(0,\delta_0) $ so that
  \begin{equation*}
    \phi(w)=(\zeta(w),w),  w\in \cD(0,\epsilon_0)
  \end{equation*}
  is a chart for $ \reg \cZ \cap \cZ_k $ around the origin.

  If $ f_{2,k} $ divides $ f_1 $, then $ f_1 $ vanishes identically on $ \cZ_k $, and its Bernstein exponent
  is $ 0 $ by convention (in any chart). So, we just need to treat the case when
  $ f_{2,k} $ and $ f_1 $ have no common factors. Let $ \psi(w)=f_1(\phi(w)) $. We
  claim that
  \begin{equation}\label{eq:Bern-claim}
    B_\psi(1/4;w_0,R)\le C(f_2) \deg f_1
  \end{equation}
  provided $ \cD(w_0,R)\subset \cD(0,\epsilon_0/8) $. We will check this claim by using the previous lemma and
  \cref{prop:bernstein-analytic}. Let $ a_1,\ldots,a_n $ be the zeros of $ \psi $.
  Since $ f_1 $ and $ f_{2,k} $
  are co-prime, using the classical B\'ezout theorem, we have
  \begin{equation*}
    n\le \deg f_{2,k}\times \deg f_1.
  \end{equation*}
  Factorize
  \begin{equation*}
    \psi(w)=h(w)P(w),\quad P(w)=\prod_{k=1}^n(w-a_k).
  \end{equation*}
  From the proof of
  \cref{prop:bernstein-analytic} (see \cref{eq:Bern-h-P}) we have
  \begin{equation*}
    B_\psi(1/4;w_0,R)\le CR(M_h(0,3\epsilon_0/4)-M_h(0,\epsilon_0/4))+B_P(1/4;w_0,R).
  \end{equation*}
  Recall that $ B_P(1/4;w_0,R)\le n\log 4\le C(f_2)\deg f_1 $. So, to check the claim \cref{eq:Bern-claim}
  we just need to estimate $ M_h(0,3\epsilon_0/4)-M_h(0,\epsilon_0/4) $. Without loss of generality we can
  assume $ h(0)=1 $ and therefore $ M_h(0,\epsilon_0/4)\ge 0 $.
  Since $h $ does not vanish we have
  \[
    h(w)=e^{u(w)+iv(w)},
  \]
  where $u+iv$ is analytic and $ u,v $ are real-valued. Then
  \begin{equation*}
    M:=M_h(0,3\epsilon_0/4)=\sup_{w\in \cD(0,3\epsilon_0/4)}|u(w)|.
  \end{equation*}
  Due to the Borel-Carath\'eodory estimate (see \cite[Thm.~11.1.1]{Lev96})
  \begin{equation*}
    N:=\sup_{w\in \cD(0,7\epsilon_0/8)} |v(w)|\gtrsim M.
  \end{equation*}
  Choose $|\hat w|= 7\epsilon_0/8$ such that $|v(\hat w)|\ge N/2$ and at the same time no root
  $a_k$ falls on the straight line  through $\hat w$ and the origin. This allows us to define the
  continuous functions
  $\theta_k(\xi):=\arg (\xi \hat w-a_k)\in [0,2\pi]$, $\xi\in (-\infty,+\infty)$. Set
  \[
    \theta(\xi)=\sum_{1\le k\le n} \theta_k(\xi).
  \]
  Take $ \theta\in (0,2\pi) $ arbitrary. We have
  \begin{equation*}
    \Im e^{-i\theta}\psi(\xi \hat w )=e^{u(\xi \hat w)}|P(w)|\sin (v(\xi \hat w)+\theta(\xi)-\theta).
  \end{equation*}
  It is clear form this formula that if $ N\gg n $, then for any $ \theta $,
  \begin{equation*}
    \# \{ \xi\in (-7\epsilon_0/8,7\epsilon_0/8) : f_1(\zeta(\xi \hat w),\xi \hat w)\in \cL_{\theta} \}
    \ge N/4,
  \end{equation*}
  where $ \cL_{\theta} $ is the line of angle $ \theta $ through the origin.
  This and Lemma~\ref{lem:bezoutpuiseoux1} imply that we must have
  \[
    N\lesssim (\deg f_{2,k})^2 \deg f_1.
  \]
  Putting the above together we have
  \begin{equation*}
    M_h(0,3\epsilon_0/4)-M_h(0,\epsilon_0/4)\lesssim C(f_2) \deg f_1,
  \end{equation*}
  which completes the proof of claim \cref{eq:Bern-claim}.

  Finally, it is clear that the conclusion holds by choosing the charts to be rescaled versions of
  $ \phi\vert_{\cD(0,\epsilon_0/8)} $, for each $ (z_0,w_0)\in \reg \cZ $.
\end{proof}

\def\cprime{$'$}

\end{document}